\documentclass[a4paper]{amsart}

\usepackage{epsfig}
\usepackage{amsthm,amsfonts}
\usepackage{amssymb,graphicx,color}
\usepackage[all]{xy}
\usepackage{verbatim}
\usepackage{hyperref}
\usepackage{enumitem}

\newtheorem{theorem}{Theorem}[section]
\newtheorem*{theorem*}{Theorem}
\newtheorem{lemma}[theorem]{Lemma}
\newtheorem{corollary}[theorem]{Corollary}
\newtheorem{proposition}[theorem]{Proposition}

\newtheorem{remark}[theorem]{Remark}

\DeclareMathOperator{\Div}{div}
\DeclareMathOperator{\ric}{Ric}

\begin{document}

\title[Foliations by spacelike hypersurfaces on Lorentz manifolds]{Foliations by spacelike hypersurfaces on Lorentz manifolds}

\author[Euripedes Carvalho da Silva]{Euripedes Carvalho da Silva}
\address[Euripedes C. Silva]{Instituto Federal de Educa\c c\~ao, Ci\^encia e Tecnologia do Cear\'a, 
Avenida Parque Central, 1315, CEP 61939-140, Maracana\'u, Cear\' a, Brazil}
\email{euripedescarvalhomat@gmail.com} 

\author[Rosa Maria dos Santos Barreiro Chaves]{Rosa Maria dos Santos Barreiro Chaves}
\address[Rosa M. B. Chaves]{Departamento de Matem\'atica - Universidade de S\~ao Paulo, Rua do Mat\~ao, 1010, CEP 05508-090, S\~ao Paulo, S\~ao Paulo, Brazil}
\email{rosab@ime.usp.br}

%
\thanks{Partially supported by Coordination for the Improvement of Higher Level -or Education- Personnel (CAPES), Brazil}

\begin{abstract}
In this work, we study the geometric properties of spacelike foliations by hypersurfaces on a Lorentz manifold. 
We find an equation that relates the foliation with the ambient manifold and applies it to investigate conditions for the leaves being totally umbilical or geodesic. Using the Maximum principle with the mentioned equation we obtain an obstruction for the existence of 
totally geodesic spacelike foliations in a spacetime with positive Ricci curvature on the direction $N$.
\end{abstract}

\maketitle

\section{Introduction}

The notion of a foliation by hypersurfaces on Riemannian manifolds have been studied by many authors and for our purposes, we point out the works  \cite{lucas, gomes, montielriem}. They focus on the geometry of the leaves in order to answer if the leaves are totally geodesic, umbilical or stable hypersurfaces,  between some other properties. The totally geodesic foliations were studied in  \cite{abe, lucas, gomes}. We highlight the work of Barbosa et. al in \cite{lucas} where they proved that a foliation whose leaves have constant mean curvature should be totally geodesic, and such a foliation does not exist in the sphere. 

The analysis of umbilicity can be found in \cite{gomesumb, montielriem}. Gomes \cite{gomesumb} shows that the existence of an umbilical leaf in a Riemannian manifold of constant sectional curvature implies the umbilicity of the foliation, that is, all leaves are umbilical hypersurfaces. Also, Montiel \cite{montielriem}, shows that the existence of a closed conformal vector field implies the existence of an umbilical foliation.

A Lorentz manifold $(M,g_{-1})$ of dimension $n+1$ is a smooth manifold endowed  with  a pseudo-Riemannian metric $g_{-1}$ of signature $(+ + \cdots + -)$. So, tangent vectors $X_p \in T_pM$ are called spacelike, lightlike or timelike if $g_{-1}(X_p,X_p)$ is positive, equal to zero or negative, respectively. We say that a Lorentz manifold $M$ is time-oriented if there is a vector field $X \in \Gamma(TM)$ such that at all points $p \in M$, the vector $X_p$ is timelike. A hypersurface $L$ of a Lorentz manifold $M$ is said to be spacelike if the metric on $L$, induced by the metric of $M$, is positive definite, that is, the induced metric on $L$ is Riemannian.

The interest of studying spacelike foliations by hypersurfaces on a Lorentz manifold is well known not only from the mathematical point of view but also from the physical one, because of their role in different problems in General Relativity (see \cite{aliasgervasio, barbosalorentz, henrique, montiel}). In  1993, Barbosa and Oliker \cite{barbosalorentz} considered the same problem for spacelike hypersurfaces with constant mean curvature in the Lorentzian context and they proved that such hypersurfaces are also the critical point of the area functional for spacelike variations that preserve volume. In 1999 Montiel \cite{montiel}, studied spacelike foliations on  Lorentz manifolds and showed that the existence of a timelike closed and conformal vector field on such a manifold furnishes a totally umbilical spacelike foliation.  

Taking into account the motivations given above, the following questions appear naturally: {\it For a given foliation by spacelike hypersurfaces on a Lorentz manifold what conditions should be imposed on the leaves to be umbilical or totally geodesic? Are there some obstructions for the existence of a  foliation by spacelike hypersurfaces on a Lorentz manifold? What conditions imply the existence of such a foliation? } \\
In this work, we analyze and answer the questions above. 

In Section $2$ we state some preliminaries.

In Section  $3$, we find a key equation that relates the principal curvatures of the leaves and the curvature of the ambient space. 
Denoting by $\mathbb{Q}^{n+1}_1(c)$ a spacetime with constant sectional curvature $c$ which is timelike geodesically complete, we answer the first question by proving the following result.
{\it Let  $\mathcal{F}$ be a foliation by spacelike hypersurfaces of $\mathbb{Q}^{n+1}_1(c)$. If the normal field to the leaves of $\mathcal{F}$ is geodesic and if there exists a totally umbilical leaf of $\mathcal{F}$ then all  leaves of $\mathcal{F}$ are totally umbilical.} 

In Section $4$ we use the key equation mentioned above to get a differential equation that is crucial to answering the other questions proposed. It appears in the result below. \\
{\it Let $\mathcal{F}$ be a spacelike foliation by hypersurfaces on a Lorentz manifold $M$ and let $N$ be a unit timelike vector field normal to the leaves of the foliation $\mathcal{F}$ in some open set $U$ of $M$. Then on $U$ we have
	\begin{equation*} 
		\Div_L(\nabla_{N}N)=nN(H)-\|\mathcal{B}\|^2+n\ric(N)+\|\nabla_NN\|^2,
	\end{equation*}
	where $\nabla$ is the Levi-Civita connection of $M$, $H$ is the mean curvature of the foliation $\mathcal{F}$ and $\|\mathcal{B}\|^2$ is the squared norm of the second fundamental form of the leaves of $\mathcal{F}$.}

Now we define the number $\mathfrak{G}_{\mathcal{F}}$ by:    
\begin{eqnarray*} \label{eq1.1}
	\mathfrak{G}_{\mathcal{F}}= \inf_{M} \left\{ \frac{1}{n} {\rm div}_{\mathcal{F}}(\nabla_{N}N)-{\rm Ric}(N)-\frac{2}{n}\|\nabla_NN\|^2\right\}.
\end{eqnarray*}
Using this number we estimate the mean curvature function and we characte\-ri\-ze spacelike foliations on Lorentz manifolds. 

{\it Let $\mathcal{F}$ be a spacelike foliation by hypersufaces on a timelike geodesically complete Lorentz manifold $M$. Then 
	\begin{align*}
		\mathfrak{G}_{\mathcal{F}} \leq 0.\\
		H^2 \leq -\mathfrak{G}_{\mathcal{F}}
	\end{align*}
	and
	$$\mathcal{F} \ \mbox{is totally geodesic} \Leftrightarrow \mathfrak{G}_{\mathcal{F}}=0.$$}

Now  let us define the anti-de Sitter space. Let $\mathbb{R}_2^{n+2}$ denote an $(n+2)$-dimensional real vector space endowed with an inner product of index $2$ given by $g_{-2} (x,y)= - \sum_{i=1}^2 {x_i}{y_i} + \sum_{j=3}^{n+2}  {x_j}{y_j}$, where $x=(x_1, x_2, \cdots x_{n+2})$ and $y=(y_1, y_2, \cdots y_{n+2})$. The anti-de Sitter space is given by the set 
$$\mathbb{H}_1^{n+1} = \{ x \in \mathbb{R}_2^{n+2} : g_{-2}(x,x) = c, \ \mbox{where} \ \ c<0 \}.$$
It is known that $\mathbb{H}_1^{n+1}$ has constant sectional curvature $c$. 

As a corollary of the previous result, we obtained also the following result, which is an obstruction to the existence of spacelike foliations in the anti de Sitter space.

{\it Let $\mathcal{F}$ be a spacelike foliation by hypersurfaces of a complete Lorentz manifold $M$ with constant  sectional curvature $c$ and with a timelike vector field $N$ normal to the leaves of $\mathcal{F}$. If $N$ is a geodesic flow, then
	\begin{enumerate}
		\item $c\geq 0$;
		\item $H^2\leq c$.
\end{enumerate}}

In Section  $5$, we define $\mathcal{L}^1(M)$ to be the space of Lebesgue integrable functions on a Riemannian manifold $M$. We say that $\|X\| \in \mathcal{L}^1(\mathcal{F})$ if and only if $\|X^{\top}\| \in \mathcal{L}^1(L)$ for each leaf $L$ of the foliation $\mathcal{F}$. Using equation \eqref{divL}, we get an obstruction for the existence of a spacelike foliation by hypersurfaces on a Lorentz manifold. 

{\it  Let $M$ be a time-oriented Lorentz manifold with negative Ricci curvature. Let $\mathcal{F}$ be spacelike foliation by hypersurfaces of $M$ with $\| \nabla_N N\| \in \mathcal{L}^1(\mathcal{F})$. Then $\mathcal{F}$ cannot have all leaves totally geodesic with one being complete. }

Let $\mathbb{R}^{n+2}$ be a $(n+2)$-dimensional Eucledian space. The Lorentz-Minkowski space $\mathbb{R}^{n+2}_1=(\mathbb{R}^{n+2},g_{-1})$, where 
$$g_{-1}(v,w)=\sum_{i=1}^{n+1}{v_iw_i-v_{n+2}w_{n+2}},$$
$v=(v_1,...,v_{n+2})$ and $w=(w_1,...,w_{n+2})\in \mathbb{R}^{n+2}$. The de Sitter space is the hyperquadric in $\mathbb{R}^{n+2}_1$ given by $\mathbb{S}^{n+1}_1=\{x\in \mathbb{R}^{n+2}_1; g_{-1}(x,x)=c, c>0\}$ which is a Lorentz manifold of constant sectional curvature $c$, with respect to the induced metric.

{\it Let $\mathcal{F}$ be a spacelike foliation by hypersurfaces on de Sitter space with $\| \nabla_N N\| \in \mathcal{L}^1(\mathcal{F})$. Then $\mathcal{F}$ cannot have all leaves totally geodesic with one being complete.}

Now we assume that $\mathfrak{G}_{\mathcal{F}}$ is finite and uses the maximum principle due to Omiri-Yua to get the next result.  

{\it Let  $\mathcal{F}$ be a spacelike foliation on a timelike geodesically complete Lorentz manifold $M$. If $L$ is a complete leaf of $\mathcal{F}$ with Ricci curvature bounded from below then there exists a sequence of points  $\{p_k\} \in L$ such that  
	
	\begin{enumerate}
		\item $\displaystyle{\lim_{k\to \infty}H_{L}(p_k)}=\sup_{L} H_L$;
		\item $\displaystyle{\lim_{k\to \infty}\|\nabla H_{L}(p_k)\|}=0$;
		\item $\displaystyle{\lim_{k\to \infty}\Delta H_L(p_k)}\leq 0$.
	\end{enumerate}
}

\section{Preliminaries}

In this section, we introduce some basic facts and notations that will appear in the paper.

Let $M$ be an $(n+1)$-dimensional Lorentz manifold endowed with a semi-Riemannian metric $g_{-1}(\cdot,\cdot)=\sum{\epsilon_A\omega_A^2}$, where  $g_{-1}(e_A,e_A)=\epsilon_A=\pm 1$ and $\mathcal{F}$ is a spacelike foliation of codimension one on $M$. 

For a given point  $p\in M$ we can choose an orthonormal frame $\{e_1,...e_n,e_{n+1}\}$ defined around  $p$ such that the vectors $e_1,...e_n$ are tangent to the leaves of  $\mathcal{F}$ and  $e_{n+1}$ is normal to them. Such a frame is usually called an adapted frame. We use the following convention of indices:
$$1 \leq A,B,C\leq n+1, \ \ \ 1 \leq i,j,k\leq n.$$

Taking the correspondent dual coframe
$$\{\omega_1,...,\omega_n,\omega_{n+1}\},$$ 
the  structure equations on  $M$ are given by 

\begin{equation}
d\omega_A=\sum_{B=1}^{n+1}{\epsilon_B\omega_B\wedge \omega_{BA}}, \ \ \omega_{AB}+\omega_{BA}=0
\end{equation}

\begin{equation}
d\omega_{AB}=\sum_{C=1}^{n+1}{\epsilon_C\omega_{AC}\wedge \omega_{CB}}+\Omega_{AB},
\end{equation}
where 
 
\begin{equation}
\Omega_{AB}=-\frac{1}{2}\sum_{C,D=1}^{n+1}{\epsilon_C\epsilon_D R_{ABCD}\omega_{C}\wedge \omega_{D}}, \ \ R_{ABCD}+R_{ABDC}=0.
\end{equation}
The Ricci curvature on the direction  $e_{n+1}$  is
\begin{equation}
\ric(e_{n+1})= -\sum_{i=1}^{n}{\langle R(e_{n+1}, e_i) e_{n+1},e_i\rangle}.   
\end{equation}

Let  $\nabla$ be the Levi-Civita connection on $M$.
Then for any tangent field $X$, we get
\begin{equation}
\nabla_X{e_A}=\sum_{B=1}^{n+1}{\epsilon_{B}\omega_{AB}(X) e_B}.
\end{equation}

Now let  $\theta_A$ and  $\theta_{AB}$ denote the restrictions of forms  $\omega_A$ and  $\omega_{AB}$ to the tangent vectors of the leaves of $\mathcal{F}$. Then it is obvious that
\begin{equation}
\theta_{n+1}=0 \ \ \mbox{e} \ \ \theta_{i}=\omega_{i},
\end{equation}

Since  $\theta_{n+1}=0$ we obtain by the structure equations  

$$0=d\theta_{n+1}=\sum_{B=1}^{n+1}{\theta_B\wedge \theta_{B n+1}}=\sum_{i=1}^n{\theta_{i}\wedge \theta_{i n+1}}.$$

By Cartan's equation, we have
\begin{equation}
\theta_{n+1i}=-\sum_{j=1}^{n}{h_{ij}\theta_j}, \ \ h_{ij}=h_{ji}.
\end{equation}

The second fundamental form  $\mathcal{B}$ of the leaves is then given by
\begin{equation}
\mathcal{B}=\sum_{i=1}^n{ \theta_i \otimes \theta_{i n+1}}=\sum_{i,j=1}^{n}{ h_{ij} \theta_i \otimes \theta_{j} }.
\end{equation}

The squared norm of $\mathcal{B} $ is defined by 
$$\left\| \mathcal {B} \right\|^2=\sum_{i,j=1}^{n} h_{ij}^{2}.$$

The mean curvature vector is  

\begin{equation}
\vec{H}= \frac{1}{n}tr(A)e_{n+1},
\end{equation}
where  $A$ is the Weingarten operator and the mean curvature function corresponding to $e_{n+1}$ is $H=\frac{-1}{n}tr(A)$.

Observe that the sign of  $H$ depends on the choice of $e_{n+1}$. The vector field defined locally by  $\vec{H}$ is globally defined on each leaf of $\mathcal{F}$. As a consequence, if  $H\neq 0$ at each point of the leaf, then the leaf is oriented. If $N$ is a unit timelike vector field normal to the leaves of $\mathcal{F}$ we can choose an adapted frame on an open set in such a way that $N=e_{n+1}$. The mean curvature of the leaf is exactly the mean curvature on the direction of $N$.

The  divergent of a vector field  $V$ defined  locally over  $M$ is defined by 
\begin{equation}
	\Div(V)=\sum_{A=1}^{n+1}{\epsilon_Ag_{-1}(\nabla_{e_A}V,e_A)}.
\end{equation}
For a vector field   tangent to the leaves of  $\mathcal{F}$ the divergent along the leaves  can be computed by
\begin{equation}
	{\Div}_L(V)=\sum_{i=1}^{n}{g_{-1}(\nabla_{e_i}V,e_i)}.
\end{equation}

\section{Umbilical foliations by spacelike hypersufaces on Lorentz manifolds}

We now turn our attention to umbilical foliations by spacelike hypersurfaces to investigate the effect of the curvature of the ambient on the leaves. We remark that, for Riemannian manifolds of constant curvature, this problem has been considered in \cite{gomesumb}.

As before, $M$ is an $(n+1)-$dimensional Lorentz manifold and $\mathcal{F}$ a spacelike foliation by hypersurfaces on $M$. Then we can choose a unit timelike vector field $N$, defined on $M$, that is normal to the leaves of $\mathcal{F}$.

\begin{proposition} 
\label{propone}
Let  $\mathcal{F}$ be a spacelike foliation by hypersurfaces on a $n+1$-dimensional  Lorentz manifold  $M$ and let $N$ be a timelike unit vector field normal to the leaves of $\mathcal{F}$. Suppose that $\{e_1,...,e_n,e_{n+1}\}$ is an adapted frame field for a neighbourhood $U\subset M$ such that $e_{n+1}=N$. Then on $U$ we get the following identity:
\begin{eqnarray*}
(dx_{i}+\sum_j{x_j\omega_{ji}})(e_{k}) & = & x_ix_k-\sum_j{h_{ij}h_{jk}}-R_{n+1 i n+1 k}\\
& - & (dh_{ik}+\sum_j{h_{jk}\omega_{ji}}+\sum_j{h_{ij}\omega_{jk}})(e_{n+1}).
\end{eqnarray*}  
where $x_i=g_{-1}(\nabla_N N,e_i)$, $1\leq i\leq n$ and $R_{n+1 i n+1 k}=g_{-1}(R(N,e_i)N,e_k) $, $1\leq i,k\leq n$.
\end{proposition}

\begin{proof}
Choose an adapted frame field on a neighbourhood of a point of $M$ such that $e_{n+1}=N$. By (2.7) it follows that 
\begin{equation} \label{wnum}
\omega_{n+1i}=-\sum_{j=1}^{n}{h_{ij}\theta_{j}}+x_i\omega_{n+1},
\end{equation}
where 
\begin{equation}
x_i=\omega_{n+1i}(N)=g_{-1}(\nabla_N N,e_i).
\end{equation}

Denoting  $X= \nabla_N N$, we get 
\begin{equation} \label{xi}
x_i=g_{-1}(X,e_i).
\end{equation}

Let us represent  the exterior derivative of the leaf  $L$ of $\mathcal{F}$ by  $d_L$. If  $f$ is a function defined on any open set of $M$, where is defined an adapted frame it follows that
\begin{eqnarray}
df=d_Lf + N (f) \omega_{n+1}.
\end{eqnarray}

From (3.1) and (2.12) we get  
\begin{eqnarray*} 
d\omega_{n+1i} & = & d(-\sum_{j}{h_{ij}\theta_{j}}+x_i\omega_{n+1})\\
& = & -\sum_{j}{(dh_{ij}\wedge\theta_{j}+h_{ij}d\theta_{j})}+dx_i\wedge\omega_{n+1}+x_i\wedge d\omega_{n+1}.
\end{eqnarray*}

Using (3.4) we get 

\begin{eqnarray*}
d\omega_{n+1i} & =  & -\sum_{j}{(d_{L}h_{ij}+ N (h_{ij})\omega_{n+1})\wedge\theta_{j}}\\
&  &- \sum_{j}{h_{ij}d\theta_{j}}+dx_i\wedge\omega_{n+1}+x_i\wedge d\omega_{n+1}.
\end{eqnarray*}

By the structure equations and using (\ref{wnum}) we can write

$$\omega_{jn+1}=\sum_{k}{h_{jk}\theta_{k}}-x_j\omega_{n+1},$$

and 
\begin{equation} \label{dxi}
dx_{i}=\sum_{k}{a_{k}\omega_{k}}+a_{n+1}\omega_{n+1},
\end{equation}
where  $a_{B}=dx_{i}(e_{B})$.  Then 
$$dx_{i}\wedge \omega_{n+1}=\sum_{k}{dx_{i}(e_{k})\omega_{k}\wedge \omega_{n+1}}.$$

By the above expression and the fact that
$$\omega_{ji}=\sum_{l=1}^{n}{a_{l}\omega_{l}}+a_{n+1}\omega_{n+1},$$
where  $a_{k}=\omega_{ji}(e_{k})$ for $1\leq k\leq n+1$, it follows that

\begin{eqnarray*}
d\omega_{n+1i} & = & \sum_{j}{N(h_{ij})\theta_{j}\wedge\omega_{n+1}}+\sum_{j,k}{h_{ij}h_{jk}\theta_{k}\wedge\omega_{n+1}}\\
& - &  \sum_{j,k}{h_{ij}\omega_{k}\wedge\omega_{kj}}+dx_i\wedge\omega_{n+1}+\sum_{k}{x_i\omega_{k}\wedge\omega_{kn+1}}\\
& + &  \sum_{j}{\theta_{j} \wedge d_{L}h_{ij}}\\
& = &  \sum_{k}{\{N(h_{ik})+\sum_j{h_{ij}\omega_{jk}(e_{n+1})}+\sum_j{h_{ij}h_{jk}}}\\
& + & {dx_{i}(e_{k})-x_ix_k\}\theta_{k}\wedge\omega_{n+1}}+\mbox{terms in} \ \omega_{k}\wedge \omega_{l}. 
\end{eqnarray*}
On the other hand from  (2.2) and  (2.3)  we obtain
\begin{eqnarray*}
d\omega_{n+1i} & = & \sum_{C}{\epsilon_C\omega_{n+1C}\wedge \omega_{Ci}}+\Omega_{n+1i}\\
& = & \sum_{k}{\omega_{n+1k}\wedge \omega_{ki}}+\sum_{k}{R_{n+1i n+1k} \omega_{n+1}\wedge \omega_{k}}\\
& - & \frac{1}{2}\sum_{k,l}{R_{n+1ikl} \omega_{k} \wedge \omega_{l}}\\
& = & \sum_{k}{\{-\sum_j{h_{jk}\omega_{ji}(e_{n+1})}}-\sum_j{x_j\omega_{ji}(e_k)} -R_{n+1 i n+1 k}\} \\
& & \omega_{k}  \wedge \omega_{n+1} +  \mbox{terms in} \   \omega_{k}\wedge \omega_{l}.
\end{eqnarray*}

After comparing the two expressions of  $d\omega_{n+1i}$ we get

\begin{eqnarray*}
(dx_{i}+\sum_j{x_j\omega_{ji}})(e_{k}) & = & x_ix_k-\sum_j{h_{ij}h_{jk}}-R_{n+1 i n+1 k}\\
& - & (dh_{ik}+\sum_j{h_{jk}\omega_{ji}}+\sum_j{h_{ij}\omega_{jk}})(e_{n+1}). 
\end{eqnarray*}  

\end{proof} 

Using Proposition \ref{propone}  we get the following result for the spacetime that also includes  Theorem 2.1 in \cite{gomesumb}. We say that spacetime is timelike geodesically complete if every maximal timelike geodesic is defined on the entire real line. Let  $\mathbb{Q}_1^{n+1}(c)$ denote a connected, timelike geodesically complete spacetime with constant sectional curvature $c$.

\begin{theorem} \label {Theo 1}
Let $\mathcal{F}$ be a spacelike foliation by hypersurfaces of $\mathbb{Q}_1^{n+1}(c)$. If the normal field to the leaves of $\mathcal{F}$ is geodesic and if there exists a totally umbilical leaf of $\mathcal{F}$ then all  leaves of $\mathcal{F}$ are totally umbilical. 
\end{theorem}

\begin{proof}
Let  $x\in \mathbb{Q}_1^{n+1}(c)$ and $\{e_1,...,e_n,e_{n+1}\}$ an adapted orthonormal frame defined  on a neighbourhood $U\subset \mathbb{Q}_1^{n+1}(c)$ such that $U \cap L\neq \emptyset$, where  $L$ is an umbilical leaf of  $\mathcal{F}$. Observe that $\{e_1,...,e_n\}$ can be chosen  in order to diagonalize the second fundamental form $\mathcal{B}$. Since the normal field  $N$ is geodesic, we obtain the following differential equation
\begin{equation}
h_{ii}^2+c+dh_{ii}(N)=0.
\end{equation}

Then the principal curvatures of the leaves satisfy the same differential equation on the integral curve $\gamma$ of $N$, on $U$. Since one leaf of $\mathcal{F}$ is  totally umbilical we have the same  initial condition for all  differential equations, when varying $i$. Using the theorem of existence and unicity of differential equations we obtain that  $h_{11}=...=h_{nn}$ on  $U$. Then the result follows from the fact that $\mathbb{Q}_1^{n+1}(c)$ is timelike geodesically complete and connected.    
\end{proof}

Let us show that the hypothesis requiring the curvature of the ambient manifold to be constant, on the previous theorem, cannot be removed. The construction of the counterexample was inspired in the paper due to Gomes \cite{gomesumb}.
\begin{remark}
	Consider $\mathbb{R}^3$ endowed with the Lorentz metric $ds^2=-dz^2+g_z(dx,dy)$, where $g_z$ is a smooth family of metrics in $\mathbb{R}^2$ depending on $z$. For our purposes consider   $g_z$ a smooth metric such that it is constant for $z \in [0,1]$ and not constant elsewhere. Then $(\mathbb{R}^3, ds^2)$  does not have constant curvature and can be foliated by the spacelike planes $z=\mbox{const}$. The vector field $N:=\frac{\partial}{\partial z}$ is unitary concerning the metric   $ds^2$  and is a geodesic field (recall that a unitary vector field has geodesic orbits if and only if its flow preserves the perpendicular plane field).  Recall that a point $(x,y,z)$  is umbilical  if and only if $\frac{\partial g}{\partial z}= \lambda g_z$. Then,  for instance, the plane $z=\frac{1}{2}$ is umbilical (actually totally geodesic), but the planes $z= k$ for $ k < 0$ or $k >1$  may not be umbilical since they may not satisfy the last condition.
\end{remark}

\section{Characterization of totally geodesic spacelike foliations on Lorentz manifolds} 
In this section, we study foliations by spacelike hypersurfaces $\mathcal{F}$ on a Lorentz manifold. We analyze the questions mentioned at the introduction of this paper, but first, we find a fundamental equation that relates the manifold with its spacelike leaves. 
\begin{proposition} \label{proptwo}
	Let $\mathcal{F}$ be a spacelike foliation by hypersurfaces on a Lorentz manifold $M$ and let $N$ be a unit field normal to the leaves of $\mathcal{F}$ on some open set $U$ of $M$. Then on $U$ we have
	\begin{equation}\label{divN}
		\Div N= nH;
	\end{equation}
	
	\begin{equation} \label{divL}
		\Div_L(\nabla_NN)=nN(H)-\|\mathcal{B}\|^2+n\ric(N)+\|\nabla_NN\|^2;
	\end{equation}
	
	\begin{equation}\label{divX}
		\Div \nabla_NN= \Div_L \nabla_NN+ \|\nabla_NN\|^2,
	\end{equation}
	where $H$ is the mean curvature on the direction  $N$.
\end{proposition}
\begin{proof}
From (\ref{dxi}) we can write
\begin{eqnarray*}
dx_{i}(e_k) & = & d_{L}x_{i}(e_k).
\end{eqnarray*}
As $\mathcal{F}$ is spacelike let we derive (\ref{xi}) with respect to  $e_{i}$ and using the fact that  $(\nabla_{e_{i}}e_{i})^T=\sum_{j}{\langle \nabla_{e_i}e_{i},e_{j}\rangle}e_{j}$ we get 
\begin{eqnarray*}
e_{i}(x_i) & = & e_{i}(g_{-1}(\nabla_NN,e_i))  =  g_{-1}(\nabla_{e_{i}}\nabla_NN,e_i)+g_{-1}(\nabla_NN,\nabla_{e_{i}}e_i);\\
\end{eqnarray*}
On the other hand, 
\begin{eqnarray*}
(\sum_{j}{x_j\omega_{ji}})(e_i) & = & \sum_{j}{g_{-1}(\nabla_NN,e_j) g_{-1} (\nabla_{e_i}e_j,e_i)}\\
& = & -g_{-1}(\nabla_NN,\sum_{j}{g_{-1}(\nabla_{e_i}e_i,e_j) e_j)}\\
& = & -g_{-1}(\nabla_NN,\nabla_{e_i}e_i).
\end{eqnarray*}
From these two equations  we conclude that
\begin{eqnarray*}
g_{-1}(\nabla_{e_{i}}\nabla_NN,e_{i}) & = & (d_{L}x_i+\sum{x_j\omega_{ji}})(e_i).
\end{eqnarray*}
It follows from Proposition \ref{propone} that 
\begin{eqnarray*}
\Div_{L}(\nabla_NN) & = & \sum_{i=1}^n{x_{i}^2}-\sum_{i,j=1}^n{h_{ij}^2}-\sum_{i=1}^n{R_{n+1 i n+1 i}}-N(\sum_{i=1}^n{h_{ii}})\\
& = & \|\nabla_NN\|^2-\left\|\mathcal{B}\right\|^2+n\ric(N)+nN(H).
\end{eqnarray*}
To proof  (\ref{divN}), we use  (2.10) and  obtain  
\begin{eqnarray*}
\Div(N) & = & \sum_{A=1}^{n+1}{\epsilon_A g_{-1}(\nabla_{e_A}N,e_A)}\\
& = & \sum_{i=1}^{n}{\theta_{n+1i}(e_i)}\\
& = & \sum_{i=1}^n{-h_{ii}\theta_{i}(e_i)}\\
& = & nH.
\end{eqnarray*}
To prove  (\ref{divX}) observe that  if  $V$ is a vector field on  $M$ tangent to the leaves of $\mathcal{F}$, then 
$$\Div V=\sum_{i=1}^n{g_{-1}(\nabla_{e_{i}}V,e_i)}-g_{-1}(\nabla_{N}V,N)=\Div_L V+g_{-1}(V,\nabla_{N}N).$$ 
Taking  $V=\nabla_{N}N$, we finish the proof.  
\end{proof} 

Let us analyze spacelike foliations $\mathcal{F}$ on a timelike geodesically complete Lorentz manifold $M$. Now we define a number $\mathfrak{G}_{\mathcal{F}}$ by:    
\begin{eqnarray*}
\mathfrak{G}_{\mathcal{F}}= \inf_{M} \left\{ \frac{1}{n} {\rm div}_{\mathcal{F}}(\nabla_{N}N)-{\rm Ric}(N)-\frac{2}{n}\|\nabla_NN\|^2\right\}.
\end{eqnarray*}
Using this number we estimate the mean curvature function.
\begin{theorem}\label{thmthree}
Let $\mathcal{F}$ be a spacelike foliation by hypersurfaces on a timelike geodesically complete Lorentz manifold $M$. Then 
\begin{align}
\label{eqeight}
\mathfrak{G}_{\mathcal{F}} \leq 0;\\
\label{eqnine}
	H^2 \leq -\mathfrak{G}_{\mathcal{F}}
\end{align}
and
$$\mathcal{F} \ \mbox{is totally geodesic} \Leftrightarrow \mathfrak{G}_{\mathcal{F}}=0.$$
\end{theorem}
\begin{proof}
Suppose on the contrary that we have $\mathfrak{G}_{\mathcal{F}}>0$ on $M$. Then
	$$\frac{1}{n} {\rm div}(\nabla_{N}N)-{\rm Ric}(N)-\frac{2}{n}\|\nabla_NN\|^2>0.$$
By Proposition \ref{proptwo}, we can say that the following inequality holds on $M$,
	$$dH(N)>H^2.$$
	Let $\gamma(s)$ be an integral curve of the unit vector field $N$. Since $M$ is timelike geodesically complete, $\gamma$ may be extended to all $\mathbb{R}$. Thus, along $\gamma$, the above inequality has the form 
	$$(H_{\gamma(s)})^{'}>(H_{\gamma(s)})^{2}, \ \ \ \forall s\in \mathbb{R}.$$
We can choose the field $N$ such that $H_{\gamma(0)}\geq 0$. Notice that such a choice does not change the expression 
$$\frac{1}{n} {\rm div}(\nabla_{N}N)-{\rm Ric}(N)-\frac{2}{n}\|\nabla_NN\|^2.$$
Therefore the following inequalities are valid for every $s>0$,
\begin{equation}
(H_{\gamma(s)})^{'}>(H_{\gamma(s)})^2>0 \ \ \mbox{e} \ \ \frac{(H_{\gamma(s)})^{'}}{(H_{\gamma(s)})^2}>1.
\label{eqthree}
\end{equation}
Now consider the real $G$ function given by 
$$G(s)=-\frac{1}{H_{\gamma(s)}}, \ \ s>0.$$
Take a fixed  $b>0$ and apply the mean value theorem for the function $G$ in the interval $[b,s]$, to obtain
$$-\frac{1}{H_{\gamma(s)}}+\frac{1}{H_{\gamma(b)}}=\frac{(H_{\gamma(\xi)})^{'}}{(H_{\gamma(\xi)}^{\alpha})^2}(s-b),$$
where $\xi \in (b,s)$. Consequently by \eqref{eqthree}, for all $b<s$, we have
$$ -\frac{1}{H_{\gamma(s)}}+\frac{1}{H_{\gamma(b)}}>s-b.$$
As $s$ tend to infinity, the right side of this inequality is unlimited while the left side is limited, which is a contradiction. This proves (\ref{eqeight}). 

To prove (\ref{eqnine}), suppose also on the contrary that there exists a point $p\in M$ such that,
$$(H_{p})^2 > - \mathfrak{G}_{\mathcal{F}}.$$
If $\mathfrak{G}_{\mathcal{F}}=-\infty$ then there is nothing to be proved. By (\ref{eqeight}), we can say that for some $a\geq 0$ we have,
$$\mathfrak{G}_{\mathcal{F}}=-a^2,$$  
where $a=0$ or $a>0$.
If $a>0$, we have from the hypothesis that $(H_p)^{2}-a^2>0.$ Let $\gamma$ be an integral curve of $N$ such that $\gamma(0)=p$.
As before we can choose a direction $N$ such that $H_p=H_{\gamma(0)}\geq 0$ and then, 
$$H_p=H_{\gamma(0)}>a.$$
By continuity, there exists a maximal interval $[0,b)$ where,  
$$(H_p)^{2}-a^2>0, \ \ \ \forall s\in [0,b).$$
We claim that $b= +\infty$.

In fact, if $b<+\infty$, by continuity we should have $(H_{\gamma(b)})^{2}=a^2$. But, from Proposition (\ref{proptwo}), we have:
$$(H_{\gamma(s)})^{'} \geq (H_{\gamma(s)})^{2}-a^2 > 0, \ \ \ \forall s\in [0,b).$$
Thus we conclude that $H_{\gamma(s)}$ is a strictly increasing function in $[0,b]$, which is a contradiction. 
Therefore, the following inequalities are valid for every $s>0$,
$$H_{\gamma(s)}>0, \ \ (H_{\gamma(s)})^{'}\geq (H_{\gamma(s)})^2-a^2>0 \ \mbox{and} \ \frac{(H_{\gamma(s)})^{'}}{(H_{\gamma(s)})^2-a^2}\geq 1.$$
Let us consider the function $L$ defined by,
$$L(s)=\frac{1}{2a}\ln(\frac{H_{\gamma(s)}-a}{H_{\gamma(s)}+a}), \ \ s>0.$$ 
For a fixed $b>0$ and using the mean value theorem, there exists $c\in [b,s]$ such that 
$$\frac{1}{2a}\ln(\frac{H_{\gamma(s)}-a}{H_{\gamma(s)}+a})-\frac{1}{2a}\ln(\frac{H_{\gamma(b)}-a}{H_{\gamma(b)}+a})=\frac{(H_{\gamma(c)})^{'}}{(H_{\gamma(c)})^2-a^2}(s-b).$$
Consequently, for all $s>b$, we have
$$\frac{1}{2a}\ln(\frac{H_{\gamma(s)}-a}{H_{\gamma(s)}+a})-\frac{1}{2a}\ln(\frac{H_{\gamma(b)}-a}{H_{\gamma(b)}+a})\geq s-b.$$
Finally let $s$ tend to infinity and get a contradiction, because the left side is limited while the right side is unlimited. The case $a=0$ is similar.

Finally, suppose that
 $$\mathfrak{G}_{\mathcal{F}}=0.$$
It follows, from equation \eqref{eqnine}, that 
$$(H_{\mathcal{F}})^2\leq -\mathfrak{G}_{\mathcal{F}}=0.$$
Therefore, $H_{\mathcal{F}}\equiv 0$ and  by Proposition \ref{proptwo} we have
$$0\geq \frac{1}{n} {\rm div}_{\mathcal{F}}(\nabla_{N}N)-{\rm Ric}(N)-\frac{2}{n}\|\nabla_NN\|^2 \geq \mathfrak{G}_{\mathcal{F}}=0.$$
Thus,  
\begin{eqnarray*}
0 & = & \frac{1}{n} {\rm div}_{\mathcal{F}}(\nabla_{N}N)-{\rm Ric}(N)-\frac{2}{n}\|\nabla_NN\|^2,
\end{eqnarray*}
on $M$. Using Proposition \ref{proptwo} we conclude that $\left\|\mathcal{B}\right\|=0$ and $\mathcal{F}$ is a totally geodesic foliation. The converse follows from Proposition \ref{proptwo}.
\end{proof}

\begin{corollary}
 Let $\mathcal{F}$ be a spacelike foliation by hypersurfaces of a complete Lorentz manifold $M$ with constant  sectional curvature $c$ and with a timelike vector field $N$ normal to the leaves of $\mathcal{F}$. If $N$ is a geodesic flow, then
 \begin{enumerate}
  \item $c\geq 0$;
  \item $H^2\leq c$.
 \end{enumerate}
\end{corollary}

\section{Some applications of the maximum principle to spacelike foliations on Lorentz manifolds}

Yau  \cite{yauomori}, established the following version of Stokes's theorem on an $n$-dimensional, complete non compact Riemannian manifold $M^n$: if $\omega \in \Omega^{n-1}(M)$ is an integrable $(n-1)$-differential form on $M^n$, then there exists a sequence $B_i$ of domains on $M$ such that $B_i\subset B_{i+1}$, 
$ M^n = \cup_{i\geq 1} B_{i}$ and 
$$\lim_{i \rightarrow +\infty}{\int_{B_i}d\omega}=0.$$
Now, suppose that $M^n$ is oriented by the volume element $dM$. If $\omega = i_XdM$ is the contraction of $dM$ in the direction of a smooth vector field $X$ on $M^n$, then Caminha et al.  \cite{caminha}, extended a result obtained by  Yau  that furnishes a version of Stokes Theorem  for complete non compact  Riemannian manifolds. They obtained a suitable consequence of Yau's result, which is described below. In what follows, $\mathcal{L}^1(M)$ stands for the space of Lebesgue integrable functions on $M^n$.
\begin{lemma}\label{lemmacaminha}
Let $X$ be a smooth vector field on the $n$-dimensional complete, non compact, oriented Riemannian manifold $M^n$, such that $\Div_MX$ does not change sign on $M^n$. If $\|X\|\in \mathcal{L}^1(M)$, then $\Div_MX=0$. 
\end{lemma}

Using their result and  equation (\ref{divL}), we get an obstruction for the existence of spacelike foliations on a Lorentz manifold. Here we say that $\|X\| \in \mathcal{L}^1(\mathcal{F})$ if and only if $\|X^{\top}\| \in \mathcal{L}^1(L)$ for each leaf $L$ of the foliation $\mathcal{F}$, where $X^\top$ means the tangent projection of field $X$ on the leaf $L$.
\begin{theorem} Let $M$ be a time-oriented Lorentz manifold with positive Ricci curvature. Let $\mathcal{F}$ be a spacelike foliation by hypersurfaces of $M$ with $\| \nabla_N N\| \in \mathcal{L}^1(\mathcal{F})$. Then $\mathcal{F}$ cannot have all leaves totally geodesic with one being complete. 
\end{theorem}
\begin{proof}
	Suppose on the contrary there exists a totally geodesic spacelike foliation $\mathcal{F}$ of $M$ with at least a complete leaf $L$. Then  $L$ is compact or $L$ is complete and noncompact. 
	If $L$ is  compact and orientable, applying equation (\ref{divL})  and Stokes  theorem we obtain
	$$0=\int_L{\Div_L(\nabla_NN)}=\int_L{n\ric(N)+\|\nabla_NN\|^2.}$$  
But this may not occur because the integral on the right side is positive.

If $L$ is complete and non  compact, as $\mathcal{F}$ is totally geodesic we get from equation  (\ref{divL}), that
$$\Div_L(\nabla_NN)=n\ric(N)+\|\nabla_NN\|^2,$$
does not changes sign. Since  $\|\nabla_NN\|$ is Lebesgue  integrable, from Lemma \ref{lemmacaminha} it follows that $\Div_L(\nabla_NN)=0$, which may not occur.
\end{proof}
As a corollary of the above result we obtain also: 

\begin{corollary}
Let $\mathcal{F}$ be a spacelike foliation by hypersufaces on the de Sitter space with $\|\nabla_N N\| \in \mathcal{L}^1(\mathcal{F})$. Then $\mathcal{F}$ cannot have all leaves totally geodesic with one being complete.
\end{corollary}

Now we assume that $\mathfrak{G}_{\mathcal{F}}$ is finite and use the well known maximum principle due to Yau \cite{yauomori}, to get:

\begin{corollary} 
Let  $\mathcal{F}$ be a spacelike foliation on a timelike geodesically complete Lorentz manifold $M$. If $L$ is a complete leaf of $\mathcal{F}$ with Ricci curvature bounded from below then there exists a sequence of points  $\{p_k\} \in L$ such that  

\begin{enumerate}
\item $\displaystyle{\lim_{k\to \infty}H_{L}(p_k)}=\sup_{L} H_L$;
\item $\displaystyle{\lim_{k\to \infty}\|\nabla H_{L}(p_k)\|}=0$;
\item $\displaystyle{\lim_{k\to \infty}\Delta H_L(p_k)}\leq 0$.
\end{enumerate}
\end{corollary}
\begin{proof}
From Theorem  \ref{thmthree} it follows that  $H_L^2$ is bounded, since  $\mathfrak{G}_{\mathcal{F}}$ is finite. Using  Yau  \cite{yauomori}, we finish the proof.
\end{proof}
\noindent {\bf Acknowledgments.}  Both authors are grateful to Antonio Caminha for interesting and useful discussions. This work was carried out during their visit in 2018 and 2019 a visit at Universidade Federal do Ceará.

\end{document}